\documentclass{amsart} 

\usepackage{tikz}

\newtheorem{theorem}{Theorem}[section]
\newtheorem{lemma}[theorem]{Lemma}
\newtheorem{corollary}[theorem]{Corollary}
\newtheorem{proposition}[theorem]{Proposition}

\newtheorem{example}[theorem]{Example}

\allowdisplaybreaks

\DeclareMathOperator{\crank}{crank}
\DeclareMathOperator{\mex}{mex}

\begin{document}

\title{Combinatorial Perspectives on the Crank and Mex Partition Statistics}

\author{Brian Hopkins}
\address{Department of Mathematics and Statistics, Saint Peter's University, Jersey City, NJ 07306, USA}
\email{bhopkins@saintpeters.edu}

\author{James A. Sellers}
\address{Department of Mathematics and Statistics, University of Minnesota Duluth, Duluth, MN 55812, USA}
\email{jsellers@d.umn.edu}

\author{Ae Ja Yee}
\address{Department of Mathematics, The Pennsylvania State University, University Park, PA 16802, USA}
\email{yee@psu.edu}



\maketitle

\begin{abstract}
Several authors have recently considered the smallest positive part missing from an integer partition, known as the minimum excludant or mex.  In this work, we revisit and extend connections between Dyson's crank statistics, the mex, and Frobenius symbols, with a focus on combinatorial proof techniques.  One highlight is a generating function expression for the number of partitions with a bounded crank that does not include an alternating sum.  This leads to a combinatorial interpretation involving types of Durfee rectangles. 
\end{abstract}

\section{Preliminaries} \label{intro}
For a positive integer $n$, a partition of $n$ is a finite sequence of integers $\lambda = \lambda_1, \lambda_2, \ldots, \lambda_r$ with $\lambda_1 \geq \lambda_2 \geq \cdots \geq \lambda_r \geq 1$ such that $\lambda_1 + \lambda_2 + \cdots + \lambda_r = n$.  Write $\ell(\lambda) = r$ for the length of $\lambda$ or number of parts and $|\lambda| = n$ for its weight or sum. Let $p(n)$ denote the number of partitions of $n$ and define $p(0) = 1$.  Also, let $q(n)$ denote the number of partitions of $n$ with distinct parts.

In 1944, Dyson \cite{dyson} suggested the existence of an integer partition statistic, which he called the crank, to combinatorially prove a divisibility property of the function $p(n)$ proven 25 years earlier by Ramanujan via generating function manipulations.  In 1988, Andrews and Garvan \cite{ag} provided a definition of the elusive crank, which we restate here.  

For a partition $\lambda$, let 
$\omega(\lambda)$ be the number of parts 1 and $\mu(\lambda)$ the number of parts greater than $\omega(\lambda)$. Then the crank of $\lambda$ is
\begin{equation}
\text{crank}(\lambda) =\begin{cases} \lambda_1 & \text{ if $\omega(\lambda)=0$},\\
\mu(\lambda) - \omega(\lambda) & \text{ if $\omega(\lambda)>0$}.
\end{cases} \label{crank}
\end{equation}

\begin{example}
$\crank(5,4,4,2,2) = 5$ and $\crank(5,4,2,2,1,1) = 2- 2=0$. 
\end{example}

In a separate work, Garvan \cite{gar} provided very useful generating function results related to the crank statistic.  Given integers $m$ and $n>1$, 
let $M(m,n)$ be the number of partitions of $n$ with crank $m$. Also, define $M(0,0)=M(1,1)=-M(0,1)=M(-1,1)=1$.  His results incorporate the standard $q$-series notation
\[(a;q)_n = (1-a)(1-aq)\cdots(1 - aq^{n-1}), \quad (a;q)_\infty = \lim_{n\rightarrow \infty} (a;q)_n.\]

\begin{theorem}[Garvan] \label{thm1.4} 
We have
\begin{align}
\sum_{m=-\infty}^{\infty} \sum_{n\ge 0} M(m,n) z^m q^n &=\frac{(q;q)_{\infty}}{(zq;q)_{\infty} (q/z;q)_{\infty}}, \label{crank1}\\
\sum_{n\ge 0} M(m,n) q^n  & =\frac{1}{(q;q)_{\infty}} \sum_{n\ge 1} (-1)^{n-1} q^{n(n-1)/2 +n |m| }(1-q^n). \label{crank2}
\end{align}
\end{theorem}
Note that it is clear from \eqref{crank1} that, for all integers $m$,
\begin{equation}
M(m,n)=M(-m,n). \label{cranksymmetry}
\end{equation}

In recent years, another integer partition statistic has arisen which is much easier to define, commonly known as the mex of a partition, from minimal excludant.  The mex of a set of integers is the smallest positive integer not in the set.  

\begin{example} 
$\mex(5,4,4,2,2) = 1$, $\mex(5,4,2,1,1) = 3$, and $\mex(4,3,2,1) = 5$. 
\end{example}

The application of the mex to partitions dates to at least 2006 with Grabner and Knopfmacher \cite{gk}; they called it the least gap.  In 2011, Andrews \cite{a11}  worked with ``the smallest part that is {\it not} a summand.''  Recently, Andrews and Newman \cite{an19, an20} began using the term mex and introduced various related statistics.
See da Silva and Sellers \cite{ds} for parity results on some of these mex-related functions. 

The initial connection between crank and mex was found independently in \cite[Theorem 2]{an20} and \cite[Theorem 1]{hs}:

\begin{theorem}[Andrews, Newman; Hopkins, Sellers] \label{cm}
The number of partitions of $n$ with nonnegative crank equals the number of partitions of $n$ with odd mex.
\end{theorem}

Further results of Hopkins and Sellers \cite{hs}, followed by Hopkins, Sellers, and Stanton \cite{hss}, establish additional connections between partitions satisfying certain crank conditions and partitions with certain mex properties.  Those proofs primarily use generating function identities.  Here we provide a combinatorial perspective on some of those results and establish new ones.  Also, recent work of Huh and Kim \cite{hk} includes results involving some of these ideas.  See the end of this section for a more detailed summary of this paper.

Here are a few more concepts we will use.  

In addition to the notation $(a;q)_n$, we need the $q$-binomial coefficients
\[ \begin{bmatrix} n \\ d \end{bmatrix} = \frac{(q;q)_n}{(q;q)_d (q;q)_{n-d}} \]
which are also known as Gaussian polynomials.  In particular, we will use the identity \cite[(I.43)]{gr}
\begin{equation}
\frac{1}{(q;q)_d}=\frac{1}{(q^{n-d+1};q)_d}
\begin{bmatrix} n \\ d
\end{bmatrix} \label{binomial}
\end{equation}
with the following combinatorial interpretation: Partitions into at most $d$ parts are in bijection with pairs of partitions where the first is a partition into parts between $n-d+1$ and $n$ and the second is a partition into at most $d$ parts all at most $n-d$.  For more details, see the material on the ``$k$th excess'' of a partition \cite[p. 50]{ab}.

The Frobenius symbol of a partition consists of two rows of strictly decreasing nonnegative integers.  Given the Ferrers diagram of a partition, the top row of the Frobenius symbol gives the number of boxes to the right of the diagonal entries and the bottom row gives the number of boxes below the diagonal entries.

\begin{example}
The partition $5,4,4,2,2$ has Frobenius symbol
\[ \begin{pmatrix} 4 & 2 & 1 \\ 4 & 3 & 0 \end{pmatrix}.\]
\end{example}

For an integer $j\ge 0$, we define the $j$-Durfee rectangle of a partition $\lambda$ to be the largest rectangle of size $d\times (d+j)$ that fits inside the Ferrers diagram of $\lambda$.  If $j=0$, then the $j$-Durfee rectangle is the well-known Durfee square (whose dimensions are the number of columns in the Frobenius symbol).

\begin{example} 
The partition $5,4,4,2,2$ has
\begin{itemize}
\item 0-Durfee rectangle (Durfee square) size $3 \times 3$,
\item 1-Durfee rectangle size $3 \times 4$,
\item 2-Durfee rectangle size $2 \times 4$,
\item 3-Durfee rectangle size $1 \times 4$, and
\item 4-Durfee rectangle size $1 \times 5$.
\end{itemize}
\end{example}

Many of our results concern the number of partitions having crank values bounded below by a given integer.  For $j \ge 0$, the next theorem uses results of Garvan to express the number of partitions $\lambda$ with $\crank(\lambda) \ge j$.  The first identity follows from \eqref{cranksymmetry} and the second from \eqref{crank2}. 

\begin{theorem} \label{thm1.2}
For an integer $j \ge 0$, 
\begin{align}
\sum_{m\ge j} \sum_{n \ge 0} M(m,n) \, q^n & =\sum_{m\ge j} \sum_{n \ge 0} M(-m,n) \, q^n \nonumber \\
& =\frac{1}{(q;q)_{\infty}} \sum_{n\ge 0} (-1)^n q^{n(n+1)/2+j(n+1)}.  \label{cranksum}
\end{align}
\end{theorem}

In Section~\ref{bigthm}, we present a generating function result that is the foundation for much of the paper.  Unlike \eqref{cranksum}, the generating function of Theorem~\ref{thm2.1} is not an alternating sum and leads more easily to combinatorial arguments.  As with several of our results, we present both analytic and combinatorial proofs.  In fact, Section 2 includes a second combinatorial proof of a lemma used in the analytic proof of Theorem~\ref{thm2.1}.  In Section~\ref{omex} we turn to the mex statistic, in particular partitions with odd mex further refined by the parity of the length and the mex modulo 4 (deriving their generating functions uses Theorem~\ref{thm2.1}).  This leads to easier proofs of some earlier results and expanded connections between partitions with certain mex characteristics and certain ranges of crank values.  
Finally, in Section~\ref{Frob} we reconsider relations between the crank and the Frobenius symbol from a combinatorial perspective, including another application of Theorem~\ref{thm2.1}.

\section{Bounded crank generating function} \label{bigthm}

Our first major result establishes another generating function for the number of partitions with crank bounded below by an arbitrary nonnegative integer.

\begin{theorem} \label{thm2.1}
For an integer $j\ge 0$, 
\begin{equation*}
\sum_{m\ge j} \sum_{n\ge 0} M(m,n) q^n =\sum_{n\ge 0} \frac{q^{(n+1)(n+j)}}{(q;q)_n (q;q)_{n+j}}.
\end{equation*}
\end{theorem}

Note that, in contrast to \eqref{cranksum} of Theorem~\ref{thm1.2}, the right-hand side here does not involve an alternating sum.  This makes it more amenable to combinatorial interpretation, as we will see in the proofs below.  A key insight is to use the symmetry \eqref{cranksymmetry} and focus on partitions with nonpositive crank which arise only from the second part of the definition \eqref{crank}.

\subsection{Analytic and combinatorial proofs of Theorem~\ref{thm2.1}}

For an analytic proof of Theorem~\ref{thm2.1}, we need the following identity given by Fine \cite[(20.51)]{f}.

\begin{lemma}[Fine] \label{lem1.6}
We have
\begin{align}
(t;q)_{\infty} \sum_{n\ge 0} \frac{t^n}{(q;q)_n (bq;q)_n} 
&=
\frac{1}{(bq;q)_{\infty}} \sum_{n\ge 0} \frac{(t;q)_n}{(q;q)_n} (-b)^n q^{n(n+1)/2} \nonumber \\
&= \sum_{n\ge 0} \frac{(bt)^n q^{n^2}}{(q;q)_n(bq;q)_n}.  \label{fine1}
\end{align}
\end{lemma}

In particular, we use the following application of Lemma~\ref{lem1.6}.

\begin{lemma} \label{lem2.2}
For $j\ge 0$,
\[ \frac{1}{(q;q)_{\infty}} \sum_{n\ge 0} (-1)^n q^{n(n+1)/2+j(n+1)}=\sum_{n\ge 0} \frac{q^{(n+1)(n+j)}}{(q;q)_n (q;q)_{n+j}}. \]
\end{lemma}

\begin{proof}[Analytic proof of Lemma~\ref{lem2.2}]
Setting $t\to q$ and $b\to q^j$ in \eqref{fine1} gives
\[ \frac{1}{(q^{j+1};q)_{\infty}} \sum_{n\ge 0} (-1)^n q^{n(n+1)/2+nj }=  \sum_{n\ge 0} \frac{ q^{n^2+n(j+1)}}{(q,q^{j+1};q)_n}. \]
Now multiply both sides
by $q^{j}/(q;q)_{j}$.
\end{proof}

\begin{proof}[Analytic proof of Theorem~\ref{thm2.1}]
The result follows directly from \eqref{cranksum} and Lemma~\ref{lem2.2}.
\end{proof}

Our combinatorial proof uses the $j$-Durfee rectangles discussed in Section~\ref{intro}.

\begin{proof}[Combinatorial proof of Theorem~\ref{thm2.1}]
By \eqref{cranksymmetry}, it is sufficient to check the generating function for partitions $\lambda$ with $\crank(\lambda) \le -j$. 

Note that for a partition $\lambda$,  if $\omega(\lambda)=0$, then $\crank(\lambda)>0$. Thus, we suppose $\omega(\lambda)>0$. Consider the $j$-Durfee rectangle, with size $d\times (d+j)$. Since $\lambda_{d}\ge d+j$, we see that  if $\omega(\lambda)< d+j$, then
$ \mu(\lambda)\ge d$, so
\begin{equation*}
\text{crank}(\lambda) =\mu(\lambda) - \omega(\lambda)> d - (d+j) =-j,
\end{equation*}
from which we see that if $\text{crank}(\lambda) \le -j$, then
$
\omega(\lambda)\ge d+j
$.
The generating function for such $\lambda$ is
\begin{equation*}
\sum_{d\ge 0} \frac{q^{d(d+j) + (d+j)}}{(q;q)_d (q;q)_{d+j}}
\end{equation*} 
where the exponent $d(d+j) +(d+j)$ of $q$ in the numerator accounts for the $j$-Durfee rectangle and the lower bound of $\omega(\lambda)$, while the factors $(q;q)_{d}$ and $(q;q)_{d+j}$ in the denominator account for the parts to the right of and below the $j$-Durfee rectangle, respectively.
\end{proof}

\subsection{Combinatorial proof of Lemma~\ref{lem2.2}}
Keeping with the combinatorial theme of the paper, we provide a combinatorial proof of Lemma~\ref{lem2.2} that uses certain triples of partitions and two forms of cancellation.

\begin{proof}[Combinatorial proof of Lemma~\ref{lem2.2}]
Setting $t\to q$ and $b\to q^j$ in Lemma~\ref{lem1.6} and multiplying by $q^j/(q;q)_{j}$, we get
\begin{align*}
(q;q)_{\infty} \sum_{n\ge 0}  \frac{q^{n+j}}{(q;q)_n (q;q)_{n+j}} &= \frac{1}{(q;q)_{\infty}} \sum_{n\ge 0}  (-1)^n q^{n(n+1)/2+j (n+1) } \\
&= \sum_{n\ge 0} \frac{ q^{(n+1)(n+j)}}{(q;q)_n(q;q)_{n+j}}.
\end{align*}
Thus, proving Lemma~\ref{lem2.2} combinatorially is equivalent to proving the following identities combinatorially:
\begin{align}
(q;q)_{\infty}  \sum_{n\ge 0}  \frac{q^{n+j}  }{(q;q)_n(q;q)_{n+j}}
&= \frac{1}{(q;q)_{\infty}} \sum_{n\ge 0}  (-1)^n q^{n(n+1)/2+j (n+1)}, \label{crank3}\\
(q;q)_{\infty} \sum_{n\ge 0}  \frac{q^{n+j}  }{(q;q)_n(q;q)_{n+j}} &= \sum_{n\ge 0} \frac{ q^{(n+1)(n+j)}}{(q;q)_n(q;q)_{n+j}}. \label{crank4}
\end{align}

Let $\mathcal{T}_j$ be the set of triples of partitions $(\pi; \kappa; \nu)$ such that $\pi$ is a partition into distinct parts, $\kappa$ is a nonempty partition with largest part at least $j$, and $\nu$ is a partition into parts that are at least $j$ less than the largest part of $\kappa$. 

\begin{example} The weight $5$ elements of $\mathcal{T}_3$ are
\[ (\emptyset; 3,2; \emptyset),  (\emptyset; 3,1,1; \emptyset),  (2; 3; \emptyset),  (1; 3,1; \emptyset),  
 (\emptyset; 4,1; \emptyset),  (1; 4;  \emptyset), (\emptyset; 4;1), (\emptyset; 5; \emptyset). \]
\end{example}

Since $\pi$ is independent of $\kappa$ and $\nu$, whereas the largest part of $\nu$ is at least $j$ less than the largest part of $\kappa$, we have
\begin{equation*}
\sum_{(\pi; \kappa; \nu) \in \mathcal{T}_j} (-1)^{\ell(\pi)} q^{|\pi |+|\kappa|+|\nu|}=\sum_{\pi} (-1)^{\ell(\pi)} q^{|\pi |}  \sum_{(\kappa;\nu)} q^{|\kappa|+|\nu|} 
\end{equation*}
where the second sum on the right-hand side is over all pairs of $\kappa$ and $\nu$ with the largest part of $\nu$ at least $j$ less than the largest part of $\kappa$.  
Thus,
\begin{equation}
\sum_{(\pi; \kappa; \nu) \in \mathcal{T}_j} (-1)^{\ell(\pi)} q^{|\pi |+|\kappa|+|\nu|}=(q;q)_{\infty} \sum_{n\ge 0} \frac{q^{n+j} }{(q;q)_n (q;q)_{n+j} }.\label{crank5}
\end{equation}

We make cancellations in $\mathcal{T}_j$ in two ways; the first cancellation will lead to \eqref{crank3}  and the second to \eqref{crank4}. 
For a triple $(\pi; \kappa; \nu)$ in $\mathcal{T}_j$, call the first largest part of $\kappa$ the peak and let the peak equal $n+j$ for some $n\ge 0$. 

$\bullet$ First cancellation: If all the parts of $\pi$ are greater than $n+j$ and $\kappa$ does not have any parts other than the peak, then we do nothing. 

If the smallest part of $\pi$ is at most the smallest part of $\kappa$, then we move the smallest part of $\pi$ to $\kappa$. If the smallest part of $\kappa$ that is not a peak is less than the smallest part of $\pi$, then we move the smallest part of $\kappa$ to $\pi$.  This process is clearly an involution. Also, it increases or decreases the number of parts of $\pi$ by $1$, so this is a sign reversing involution. Thus, after cancellation, the remaining triples in $\mathcal{T}_j$ are $(\pi; \kappa; \nu)$ where $\kappa$ has only the peak $n+j$, the parts of $\pi$ are greater than $n+j$ and distinct, and the parts of $\nu$ are at most $n$. Therefore, 
\[ \sum_{(\pi; \kappa; \nu) \in \mathcal{T}_j} (-1)^{\ell(\pi) } q^{|\pi |+|\kappa|+|\nu|}=\sum_{n\ge 0} \frac{q^{n+j} (q^{n+j+1};q)_{\infty}  }{(q;q)_n }. \]

We now make further adjustment on the remaining $(\pi; \kappa; \nu)$: Subtract $j$ from the peak, and successively subtract $j+1$ from the smallest part of $\pi$, $j+2$ from the second smallest part of $\pi$, and so on. The resulting parts of $\pi$ are at least $n$. Thus, all the parts of the resulting $\pi$, $\kappa$, and $\nu$ form an ordinary partition, and the subtracted sequence, i.e., $j, j+1, \ldots$, forms a partition into distinct parts differing by exactly $1$ with smallest part $j$.   Therefore, 
\[ \sum_{(\pi; \kappa; \nu) \in \mathcal{T}_j} (-1)^{\ell(\pi) } q^{|\pi |+|\kappa|+|\nu|}= \frac{1}{(q;q)_{\infty}} \sum_{n\ge 0} (-1)^n q^{n(n+1)/2+j (n+1)}.\]
Combining this with \eqref{crank5} completes the proof of \eqref{crank3}. 

$\bullet$ Second cancellation: We first make some adjustments on $(\pi; \kappa; \nu) \in \mathcal{T}_j$. Put the peak aside momentarily and take the $j$-Durfee rectangle of the non-peak parts of $\kappa$, which has size $d \times (d+j)$.  Since the parts of $\kappa$ are at most $n+j$, we see from the definition of the $j$-Durfee rectangle that to the right of the rectangle, there are at most $d$ parts that are at most $n-d$. By applying the bijection associated with \eqref{binomial} to those parts to the right of the Durfee rectangle with parts in $\nu$ that are between $n-d+1$ and $n$ (if they exist), we obtain a partition with at most $d$ parts. Putting this back to the right of the $j$-Durfee rectangle yields a partition with a $j$-Durfee rectangle and no restriction on part sizes. By abuse of notation, we denote by $\kappa$ the resulting partition with the $j$-Durfee rectangle and by $\nu$ the remaining parts at most $n-d$.  After dividing the peak into $d+j$ and $n-d$, add $d+j$ back to the new $\kappa$ and $n-d$ to the new $\nu$. 
Therefore, we have
\[ \sum_{(\pi; \kappa; \nu) \in \mathcal{T}_j} (-1)^{\ell(\pi) } q^{|\pi |+|\kappa|+|\nu|}
= (q;q)_{\infty} \sum_{d\ge 0} \frac{q^{d(d+j)+d+j}}{(q;q)_{d+j} (q)_d} \sum_{n\ge d} \frac{q^{n-d}}{(q;q)_{n-d} }. \]
Let $x$ be the smallest part of $\pi$ and $y$ the smallest part of $\nu$. If $x \le y$, then move $x$ to $\nu$. If $x > y$, then move $y$ to $\pi$. This is clearly a sign reversing involution. After this cancellation, the remaining triples are $(\emptyset; \kappa; \emptyset)$. Therefore, we have
\[ \sum_{(\pi; \kappa; \nu) \in \mathcal{T}_j} (-1)^{\ell(\pi) } q^{|\pi |+|\kappa|+|\nu|}= \sum_{d\ge 0} \frac{q^{d(d+j)+d+j}}{(q;q)_{d+j} (q)_d}. \]
Combining this with \eqref{crank5} completes the proof of \eqref{crank4}. 
\end{proof}

\section{Connecting mex and crank} \label{omex}

We now shift our attention to various functions related to the mex of an integer partition, ultimately with the goal of connecting such mex results with partitions satisfying certain crank conditions.  Toward this end, define $m_{a,b}(n)$ to be the number of partitions of weight $n$ with mex congruent to $a$ modulo $b$.  For example, $m_{1,2}(n)$ is the number of partitions of $n$ with odd mex.  In \cite{hss}, it was productive to split the odd mex partitions modulo 4; in our notation, these are $m_{1,4}(n)$ and $m_{3,4}(n)$.  (Note that the $m_{1,2}(n), m_{1,4}(n), m_{3,4}(n)$ are called $o(n), o_1(n), o_3(n)$, respectively, in \cite{hss}.)

Next, we show that breaking these counts down farther, by parity of partition length, shows new relations and allows for cleaner proofs of previous results.  We conclude the section with more refined connections between partitions with odd mex and those with nonpositive crank, answering a question posed in \cite{hss}.

\subsection{Refinements on odd mex by parity of length}

We will use the following lemma, which can be adapted from Ewell \cite[(6)]{e}:
\begin{equation*}
\sum_{n\ge 0} (-q)^{n(n+1)/2} = \prod_{n=1}^\infty \frac{1-q^{2n}}{1 + q^{2n-1}}.
\end{equation*}

\begin{lemma} \label{ewell}
We have
\[ \frac{1}{(q;q)_{\infty}} \sum_{n\ge 0} (-q)^{n(n+1)/2} =(-q^2;q^2)_{\infty}.\]
\end{lemma}

With this and Theorem~\ref{thm2.1}, generating functions for the initial refinements of the odd mex statistics can be found rather easily.

\begin{proposition}
We have
\begin{align}
\sum_{n\ge 0 } m_{1,4}(n) q^n & =\frac{ 1}{(q;q)_{\infty}} \sum_{k\ge 0}  q^{2k(4k+1)}(1-q^{4k+1}) \label{o1a} \\
& =\frac{1}{2} \left(\sum_{n\ge 0} \frac{q^{n(n+1)}}{(q;q)_n^2} +\sum_{n\ge 0} \frac{q^{n(n+1)}}{(q^2;q^2)_n}\right),\label{o1b} \\
\sum_{n\ge 0 } m_{3,4}(n) q^n & =\frac{ 1}{(q;q)_{\infty}} \sum_{k\ge 0}  q^{(2k+1)(4k+3)}(1-q^{4k+3}) \label{o3a} \\
& =\frac{1}{2} \left(\sum_{n\ge 0} \frac{q^{n(n+1)}}{(q;q)_n^2} - \sum_{n\ge 0} \frac{q^{n(n+1)}}{(q^2;q^2)_n}\right). \label{o3b}
\end{align}
\end{proposition}

\begin{proof}
The generating functions \eqref{o1a} and \eqref{o3a} follow directly from the definitions of $m_{1,4}(n)$ and $m_{3,4}(n)$, respectively.

The verification of the other generating functions uses Theorem~\ref{thm2.1} and Lemma~\ref{ewell}:
\begin{align*}
\sum_{n\ge 0} (m_{1,4}(n)+m_{3,4}(n)) q^n & = \sum_{n\ge 0} m_{1,2}(n) q^n \\
& = \frac{1}{(q;q)_{\infty}} \sum_{k \ge 0} q^{k(2k+1)}(1-q^{2k+1}) \\
& =\sum_{n\ge 0} \frac{q^{n(n+1)}}{(q;q)_n^2}, \\
\sum_{n\ge 0} (m_{1,4}(n)-m_{3,4}(n)) q^n&= \frac{1}{(q;q)_{\infty}} \sum_{k \ge 0} (-1)^k q^{k(2k+1)}(1-q^{2k+1}) \\
&=\frac{1}{(q;q)_{\infty}} \sum_{k\ge 0} (-q)^{k(k+1)/2}\\
&=(-q^2;q^2)_{\infty} \\
&=\sum_{n\ge 0} \frac{q^{n^2+n}}{(q^2;q^2)_n},
\end{align*}
The expressions \eqref{o1b} and \eqref{o3b} follow.
\end{proof}

The following result \cite[Proposition 9]{hss} connects $m_{1,4}(n)$ and $m_{3,4}(n)$.

\begin{proposition}[Hopkins, Sellers, Stanton] \label{o13}
For any $n\ge 1$,
\[m_{1,4}(n) = \begin{cases} m_{3,4}(n) & \text{if $n$ is odd,} \\ m_{3,4}(n) + q(n/2) & \text{if $n$ is even.} \end{cases} \]
\label{comb}
\end{proposition}

A different proof than the one given in \cite{hss} will follow from Theorem~\ref{4ways} as detailed below.

Here, we consider further refinements of these statistics incorporating the parity of partition length.  A superscript $o$ denotes the number of designated partitions with odd length, similarly a superscript $e$ for even length.  For instance, $m_{3,4}^o(n)$ is the number of partitions of $n$ with mex congruent to 3 modulo 4 and an odd number of parts, while $q^e(n)$ is the number of partitions of $n$ into an even number of distinct parts.
Values of some of these statistics for small values of $n$ are given in Table \ref{tab}.  The $m_{1,2}(n)$ sequence matches \cite[A064428]{o}; at the time of writing, no other rows are currently in that encyclopedia. 

\begin{table}
\caption{Values of various refined odd mex statistics for small $n$.
}
{\renewcommand{\arraystretch}{1.5}
\begin{tabular}{c|cccccccccccccc}
$n$ & 2 & 3 & 4 & 5 & 6 & 7 & 8 & 9 & 10 & 11 & 12 & 13 & 14 & 15 \\ \hline \hline		
$m_{1,2}(n)$ &  1 & 2 & 3 & 4 & 6 & 8 & 12 & 16 & 23 & 30 & 42 & 54 & 73 & 94  \\ \hline
$m_{1,4}(n)$ & 1 & 1 & 2 & 2 & 4 & 4 & 7 & 8 & 13 & 15 & 23 & 27 & 39 & 47  \\
$m_{3,4}(n)$ & 0 & 1 & 1 & 2 & 2 & 4 & 5 & 8 & 10 & 15 & 19 & 27 & 34 & 47  \\ \hline
$m_{1,2}^o(n)$ &  1 & 1 & 2 & 2 & 3 & 4 & 6 & 8 & 11 & 15 & 21 & 27 & 36 & 47 \\
$m_{1,2}^e(n)$ & 0 & 1 &  1 & 2 & 3 & 4 & 6 & 8 & 12 & 15 & 21 & 27 & 37 & 47\\ \hline
$m_{1,4}^o(n)$ & 1 & 1 & 1 & 1 & 2 & 2 & 3 & 4 & 6 & 8 & 11 & 14 & 19 & 24\\
$m_{1,4}^e(n)$ & 0 & 0 & 1 & 1 & 2 & 2 & 4 & 4 & 7 & 7 & 12 & 13 & 20 & 23 \\
$m_{3,4}^o(n)$ & 0 & 0 & 1 & 1 & 1 & 2 & 3 & 4 & 5 & 7 & 10 & 13 & 17 & 23 \\
$m_{3,4}^e(n)$ & 0 & 1 & 0 & 1 & 1 & 2 & 2 & 4 & 5 & 8 & 9 & 14 & 17 & 24 \\
\end{tabular}}
\label{tab}
\end{table}
Several relations between these statistics follow by definition:
\begin{gather*}
m_{1,2}(n) = m_{1,4}(n) + m_{3,4}(n) = m_{1,2}^o(n) + m_{1,2}^e(n), \\
m_{1,4}(n) = m_{1,4}^o(n) + m_{1,4}^e(n),  \\
m_{3,4}(n) = m_{3,4}^o(n) + m_{3,4}^e(n), \\
m_{1,2}^o(n) = m_{1,4}^o(n) + m_{3,4}^o(n), \\
m_{1,2}^e(n) = m_{1,4}^e(n) + m_{3,4}^e(n).
\end{gather*}

We prove another relation between these statistics that will simplify several previous results and serves as a natural refinement of Proposition~\ref{o13}.  

\begin{theorem} \label{4ways}
For any $n\ge 1$, 
\begin{align*}
m_{1,4}^o(n) = \begin{cases} m_{3,4}^e(n) & \text{if $n$ is odd,} \\ m_{3,4}^e(n) + q^o(n/2) & \text{if $n$ is even;} \end{cases} \\
m_{1,4}^e(n) = \begin{cases} m_{3,4}^o(n) & \text{if $n$ is odd,} \\ m_{3,4}^o(n) + q^e(n/2) & \text{if $n$ is even.} \end{cases}
\end{align*}
\end{theorem}

\begin{proof}
By the definitions, we see that
\begin{align*}
\sum_{n\ge 0} \left(m_{1,4}^e(n) +m_{1,4}^o(n)\right) q^n&= \frac{1}{(q;q)_{\infty}} \sum_{k\ge 0} q^{2k(4k+1)}(1-q^{4k+1}),\\
\sum_{n\ge 0} \left(m_{1,4}^e(n) -m_{1,4}^o(n)\right) q^n&=\frac{1}{(-q;q)_{\infty}} \sum_{k\ge 0} q^{2k(4k+1)}(1+q^{4k+1}),
\end{align*}
and
\begin{align*}
\sum_{n\ge 0} \left(m_{3,4}^o(n) +m_{3,4}^e(n)\right) q^n&= \frac{1}{(q;q)_{\infty}} \sum_{k\ge 0} q^{(2k+1)(4k+3)}(1-q^{4k+3}),\\
\sum_{n\ge 0} \left(m_{3,4}^o(n) -m_{3,4}^e(n)\right) q^n&=\frac{1}{(-q;q)_{\infty}} \sum_{k\ge 0} q^{(2k+1)(4k+3)}(1+q^{4k+3}).
\end{align*}
Thus,
\begin{align*}
\sum_{n\ge 0} 2 \left(m_{1,4}^e(n)-m_{3,4}^o(n)\right) q^n &=\frac{1}{(q;q)_{\infty}}\sum_{k\ge 0} (-q)^{k(k+1)/2}+\frac{1}{(-q;q)_{\infty}}  \sum_{k\ge 0} q^{k(k+1)/2} \\
&=(-q^2;q^2)_{\infty} +(q^2;q^2)_{\infty}  \\
&=\sum_{n\ge 0} 2 q^e(n/2)q^n,
\end{align*}
applying Lemma~\ref{ewell} for the second equality.  Similarly,
\begin{align*}
\sum_{n\ge 0} 2 (m_{1,4}^o(n)-m_{3,4}^e(n)) q^n &=\frac{1}{(q;q)_{\infty}}\sum_{k\ge 0} (-q)^{k(k+1)/2}-\frac{1}{(-q;q)_{\infty}}  \sum_{k\ge 0} q^{k(k+1)/2} \\
&=(-q^2;q^2)_{\infty} - (q^2;q^2)_{\infty}  \\
&=\sum_{n\ge 0} 2q^o(n/2)q^n. \qedhere
\end{align*}
\end{proof}

Next, we give the more direct proof of Proposition~\ref{o13}.

\begin{proof}[Proof of Proposition~\ref{o13}.]  
Rearranging the results of Theorem~\ref{4ways} gives
\begin{align*}
m_{1,4}(n) & = m_{1,4}^o(n) + m_{1,4}^e(n) \\
& = \begin{cases} m_{3,4}^e(n) + m_{3,4}^o(n) & \text{if $n$ is odd,} \\ m_{3,4}^e(n) + q^o(n/2) + m_{3,4}^o(n) + q^e(n/2) & \text{if $n$ is even} \end{cases} \\
& = \begin{cases} m_{3,4}(n) & \text{if $n$ is odd,} \\ m_{3,4}(n) + q(n/2) & \text{if $n$ is even.} \end{cases} \qedhere
\end{align*}
\end{proof}

Theorem~\ref{4ways} also leads to the following relation between $m_{1,2}^o(n)$ and $m_{1,2}^e(n)$; we provide both analytic and combinatorial proofs.

\begin{corollary} \label{oe}
\begin{equation*}
m_{1,2}^o(n)=\begin{cases} m_{1,2}^e(n)+(-1)^{m+1} & \text{ when $n=m(3m\pm 1)$},\\
m_{1,2}^e(n) & \text{ otherwise}.
\end{cases}
\end{equation*}
\end{corollary}

\begin{proof}[Analytic proof of Corollary~\ref{oe}]
Rearranging the second equation of Theorem~\ref{4ways} as
\[m_{3,4}^o(n) = \begin{cases} m_{1,4}^e(n) & \text{if $n$ is odd,} \\ m_{1,4}^e(n) - q^e(n/2) & \text{if $n$ is even} \end{cases}\]
allows us to write
\begin{align*}
m_{1,2}^o(n) & = m_{1,4}^o(n) + m_{3,4}^o(n) \\
& = \begin{cases} m_{3,4}^e(n) + m_{1,4}^e(n) & \text{if $n$ is odd,} \\ m_{3,4}^e(n) + q^o(n/2) + m_{1,4}^e(n) - q^e(n/2) & \text{if $n$ is even} \end{cases} \\
& = \begin{cases} m_{1,2}^e(n) & \text{if $n$ is odd,} \\ m_{1,2}^e(n) + q^o(n/2) - q^e(n/2) & \text{if $n$ is even} \end{cases}
\end{align*}
and the result follows from Euler's pentagonal number theorem.
\end{proof}

Our combinatorial proof of Corollary~\ref{oe} uses an equivalent generating function formulation.  One step incorporates the following result of Carlitz \cite{c}.

\begin{lemma}[Carlitz] \label{car}  We have
\[ \prod_{n=1}^\infty (1-x^ny^n)(1+x^ny^{n-1})(1+x^{n-1}y^n) = \sum_{n=-\infty}^{\infty} x^{n(n+1)/2} y^{n(n-1)/2}.\]
\end{lemma}

The right-hand side of this identity is now known as Ramanujan's theta series.  In terms of producing a purely combinatorial argument for Corollary~\ref{oe}, note that Wright \cite{w} gave a combinatorial verification of Lemma~\ref{car}.

\begin{proof}[Combinatorial proof of Corollary~\ref{oe}.]
From the definitions of $m^o_{1,2}(n)$ and $m^e_{1,2}(n)$, a generating function statement of Corollary~\ref{oe} is
$$
\sum_{k\ge 0} \frac{q^{k(2k+1)}}{(-q;q^2)_k (-q^{2k+3};q^2)_{\infty}} = (q^2;q^2)_{\infty}
$$
which is equivalent to 
$$
\frac{1}{(-q;q^2)_{\infty}}  \sum_{k\ge 0} q^{k(2k+1)} (1+q^{2k+1}) =   (q^2;q^2)_{\infty}.
$$
Note that the left-hand side can be rewritten
 \begin{equation}
\frac{1}{(-q;q^2)_{\infty}}  \sum_{k\ge 0} q^{k(2k+1)} (1+q^{2k+1}) =\frac{(q^2;q^2)_{\infty}}{ (-q;q^2)_{\infty}}  \frac{1}{(q^4;q^4)_{\infty} }\sum_{k =-\infty}^{\infty} q^{k(2k+1)} . \label{eq1}
\end{equation}
Now Lemma~\ref{car} with $x = q^3$ and $y = q$ gives
$$
 \frac{1}{(q^4;q^4)_{\infty} }\sum_{k =-\infty}^{\infty} q^{k(2k+1)} =\prod_{n=1}^\infty (1+q^{4n-1})(1+q^{4n-3}) =  \sum_{\nu} q^{|\nu|}
 $$
where the last sum is over all partitions into distinct odd parts. Thus, the right-hand side of \eqref{eq1} is the weighted generating function of triples $(\pi; \mu; \nu)$ where
$\pi$ is a partition into distinct even parts,
$\mu$ is a partition into odd parts, and
$\nu$ is a partition into distinct odd parts.
That is,
$$
\frac{(q^2;q^2)_{\infty}}{ (-q;q^2)_{\infty}}  \frac{1}{(q^4;q^4)_{\infty} }\sum_{k =-\infty}^{\infty} q^{k(2k+1)} =\sum_{(\pi; \mu; \nu)} (-1)^{\ell(\pi)+\ell(\mu)} q^{|\pi|+|\mu|+|\nu|}.
$$

To complete the proof, we cancel out many of the triples $(\pi; \mu; \nu)$.

Let $x$ be the smallest part of $\mu$ and $y$ the smallest part of $\nu$. If $x<y$, then move $x$ to $\nu$. Otherwise, move $y$ to $\mu$. This is clearly a sign reversing involution with no fixed points. Thus, after cancellations, we are left with only $\pi$. 

Now apply Franklin's bijection \cite[Theorem 1.6]{ab} to $\pi_1/2, \ldots, \pi_r/2$ for each $\pi$.  This cancels out partitions into an even number of parts and an odd number of parts except for exactly one partition $\pi$ of $n$ when $n$ is twice a generalized pentagonal number.
\end{proof}

Corollary~\ref{oe} gives the most succinct proof yet of \cite[Theorem 1.2]{an19}, that $m_{1,2}(n)$ is almost always even (see also \cite[Theorem 10]{hss}).

\begin{theorem}[Andrews, Newman]
$m_{1,2}(n)$ is almost always even and is odd exactly when $n=m(3m\pm1)$ for some $m$.
\end{theorem}

\begin{proof}
Using Corollary~\ref{oe}, 
\begin{align*}
m_{1,2}(n) & = m_{1,2}^o(n) + m_{1,2}^e(n) \\
& =\begin{cases} 2m_{1,2}^e(n)+(-1)^{m+1} & \text{ when $n=m(3m\pm 1)$},\\
2m_{1,2}^e(n) & \text{ otherwise}. \end{cases} \qedhere
\end{align*}
\end{proof}

Note that, for odd $n$, many of the statistics introduced in this section are equal.  Specifically,
\begin{equation}
m_{1,4}(2k+1) = m_{3,4}(2k+1) = m_{1,2}^o(2k+1) = m_{1,2}^e(2k+1) \label{oddstats}
\end{equation}
for all integers $k \ge 0$.  It would be nice to have combinatorial proofs of the identities in \eqref{oddstats}.

\subsection{Connecting odd mex and nonpositive crank}
Given Theorem~\ref{cm} which connects partitions with odd mex and partitions with nonnegative crank, the split of the odd mex partitions leads to a natural question, posed in \cite{hss}: Which partitions of $n$ with nonnegative crank correspond to the partitions counted by $m_{1,4}(n)$, and which to those counted by $m_{3,4}(n)$?  This was answered recently by Huh and Kim \cite{hk}, whose Proposition 3.4 is the even case of the following theorem.

Let $M_{\le 0}(n)$ be the number of partitions $\lambda$ of $n$ with $\crank(\lambda) \le 0$.  Using this notation, we know from Theorem~\ref{cm} and \eqref{cranksymmetry} that $$M_{\le 0}(n) = m_{1,2}(n).$$  Let $M^e_{\le 0}(n)$ be the number of partitions $\lambda$ of $n$ with $\crank(\lambda) \le 0$ having even length, similarly $M^o_{\le 0}(n)$ for odd length.

\begin{theorem}[Huh, Kim] \label{newo1o3}
We have $M^e_{\le 0}(n) = m_{1,4}(n)$ and $M^o_{\le 0}(n) = m_{3,4}(n)$.
\end{theorem}

\begin{proof}
Note that the generating function for the number $M_{\le 0} (k, n)$ of partitions $\lambda$ of $n$ into $k$ parts with $\crank(\lambda) \le 0$ is
\begin{equation*}
\sum_{k, n\ge 0} M_{\le 0} (k,n) z^k q^n=\sum_{n\ge 0}  \frac{ z^{2n} q^{n(n+1)} }{(zq;q)_n (q;q)_n}.
\end{equation*}
Substituting $z=-1$ gives
\begin{align*}
\sum_{ n\ge 0} \left(M^{e}_{\le 0} (n)- M^{o}_{\le 0}(n)\right) q^n= \sum_{n \ge 0} \frac{ q^{n(n+1)} }{(-q;q)_n (q;q)_n} =\sum_{n\ge 0} \frac{q^{n(n+1)}}{(q^2;q^2)_n}.
\end{align*}
Thus, the generating function for the number of partitions $\lambda$ of $n$ into an even number of parts with $\crank(\lambda)\le 0$ is
\begin{equation*}
\frac{1}{2}\left(  \sum_{n \ge 0}  \frac{ q^{n(n+1)} }{(q;q)_n^2} + \sum_{n\ge 0} \frac{q^{n(n+1)}}{(q^2;q^2)_n}\right). 
\end{equation*}
The theorem then follows easily with \eqref{o1b} and \eqref{o3b}.
\end{proof}

The interested reader will want to compare our proof with that of \cite{hk}.

Theorem~\ref{newo1o3} and Proposition~\ref{o13} immediately give the following corollary.  We provide a combinatorial verification, similar to the second cancellation of the combinatorial proof of Lemma~\ref{lem2.2}.

\begin{corollary}
For any $n\ge 1$,
\[M^e_{\le 0}(n) = \begin{cases} M^o_{\le 0}(n) & \text{if $n$ is odd,} \\ M^o_{\le 0}(n) + q(n/2) & \text{if $n$ is even.} \end{cases} \]
\end{corollary}

\begin{proof}
We construct a sign reversing involution on $M_{\le 0}(n)$

Suppose $\lambda \in M_{\le 0}(n)$ has Durfee square size $d\times d$. Since $\crank(\lambda)\le 0$, there are at least $d$ parts $1$.  Let $\pi$ be the partition consisting of parts below the Durfee square excluding $d$ parts $1$. Let $\nu$ be the conjugate of the partition consisting of parts to the right of the Durfee square. 

Let $x$ be the smallest part of $\pi$ appearing an odd number of times and $y$ be the smallest part of $\nu$. If  $x\le y$, then move a part of size $x$ to $\nu$.  If $x>y$, then we move a part of size $y$ to $\pi$. This decreases or increases $\ell(\pi)$ by one, so it indeed decreases or increases $\ell(\lambda)$ by one. Hence, it is a sign reversing involution. 

Partitions $\lambda$ with each part in $\pi$ appearing an even number of times and $\nu$ being the empty partition remain unchanged by the involution.  Note that this occurs only when $n = |\lambda|$ is even.  Thus 
\begin{align*}
\sum_{n\ge 0} \left(M^{e}_{\le 0} (n)- M^{o}_{\le 0}(n)\right) q^n &=\sum_{\lambda} (-1)^{\ell(\lambda)} q^{|\lambda|} \\
&= \sum_{d\ge 0} \sum_{\pi, \nu} (-1)^{\ell(\pi)} q^{d^2+d+|\pi | +|\nu|} \\
&=\sum_{d\ge 0}\frac{q^{d^2+d}}{(q^2;q^2)_d}
\end{align*}
where the second equality follows from the decomposition of $\lambda$ with Durfee square size $d\times d$, $d$ parts $1$, $\pi$ and $\nu$, and the last equality follows from the sign involution. 

Finally, 
\[\sum_{d\ge 0}\frac{q^{d^2+d}}{(q^2;q^2)_d} = (-q^2; q^2)_\infty\]
which has the following combinatorial proof: Rearrange the $d^2$ boxes of the Durfee square and $d$ parts $1$ in rows of length $2d, 2d-2, \ldots, 2$, then add the conjugate of $\pi$ to these $d$ consecutive even parts.  This produces a partition of $n$ into distinct even parts, and there are $q(n/2)$ such partitions.
\end{proof}

We hope that this proof contributes to combinatorial verifications of Theorem~\ref{cm} and its refinement, Theorem~\ref{newo1o3}, which have eluded us so far.

\section{Frobenius symbols and crank} \label{Frob}
In this final section, we revisit a theme begun by Andrews in 2011 \cite{a11}, the relationships between partitions whose Frobenius symbols satisfy certain restrictions and partitions with certain crank or mex characteristics.  In keeping with the theme of this paper, we provide combinatorial proofs for two results of Hopkins, Sellers, and Stanton.  The first is \cite[Proposition 7]{hss}.  

\begin{proposition}[Hopkins, Sellers, Stanton] \label{Frob1}
The number of partitions of $n$ with crank $0$ equals the number of partitions of $n$ whose Frobenius symbol has no $0$ minus the number of partitions of $n-1$ whose Frobenius symbol has no $0$.
\end{proposition}

\begin{proof}
Let $\lambda$ be a partition with $\crank(\lambda)=0$. If $\omega(\lambda)=0$, then $\crank(\lambda)>0$, which is a contradiction. So assume $\omega(\lambda)>0$ and let the size of the Durfee square of $\lambda$ be $d\times d$. If $d<\omega(\lambda)$, then $\mu(\lambda)\le d$ since $\lambda_{d+1}\le d <\omega(\lambda)$, so $\crank(\lambda)< 0$. Similarly, we can check that if $d>\omega(\lambda)$, then $\crank(\lambda)>0$. Hence, if $\crank(\lambda)=0$, then its Durefee square must be of size $d\times d$ with $d=\omega(\lambda)$. Also, the first $d$ parts of $\lambda$ must be greater than $d$ since, if $\lambda_{d}\le d$, then 
$\mu(\lambda)\le d-1$, so 
\[ \crank(\lambda)=\mu(\lambda)-\omega(\lambda)\le d-1-d\le -1. \]
Thus, the generating function for partitions with crank 0 is
\[ 1+\sum_{d=1}^{\infty} \frac{q^{d^2+2d}}{(q;q)_d (q^2;q)_{d-1}}. \]
A Ferrers diagram of this type is shown in the left-hand side of Figure \ref{fig1}.  To produce the partition on the right-hand side, delete the $d$ parts  $1$ and create a row of size $d$ just below the Durfee square.

\begin{figure}[b]%[th]
\begin{tikzpicture}[scale=0.8]
\draw[thick, black] (0, 1) -- (4, 1);
\draw[thick, black] (0, -0.5) -- (2, -0.5);
\draw[thick, black] (0, 1) -- (0, -4);
\draw[thick, black] (0, -4) -- (0.5, -4);
\draw[thick, black] (0.5, -4) -- (0.5, -2.5);
\draw[thick, black] (2, 1) -- (2, -0.5);
\draw[thick, black] (1.5, 1) -- (1.5, -0.5);
\draw[thick, black] (0, -2.5) -- (1, -2.5);
\draw[thick, dotted, black] (1,-2.5) -- (1.5, -0.5);
\draw[thick, dotted, black] (2,-0.5) -- (3, -0.25);

\draw (-0.3, 0.2) node{\footnotesize{$d$}};
\draw (-0.3, -3.3) node{\footnotesize{$d$}};
\draw (0.8, 1.3) node{\footnotesize{$d$}};

\draw (5,-0.5) node{$\longrightarrow$};

\draw[thick, black] (7, 1) -- (11, 1);
\draw[thick, black] (7, -0.5) -- (9, -0.5);
\draw[thick, black] (7, 1) -- (7, -3);
\draw[thick, black] (7, -1) -- (8.5, -1);
\draw[thick, black] (8.5, -0.5) -- (8.5, -1);
\draw[thick, black] (9, 1) -- (9, -0.5);
\draw[thick, black] (8.5, 1) -- (8.5, -0.5);
\draw[thick, black] (7, -3) -- (8, -3);
\draw[thick, dotted, black] (8,-3) -- (8.5, -1);
\draw[thick, dotted, black] (9,-0.5) -- (10, -0.25);

\draw (-0.3, 0.2) node{\footnotesize{$d$}};
\draw (-0.3, -3.3) node{\footnotesize{$d$}};
\draw (0.8, 1.3) node{\footnotesize{$d$}};

\end{tikzpicture}

\caption{Mapping a crank 0 partition to a partition with no 0 in its Frobenius symbol.  }
\label{fig1}
\end{figure}
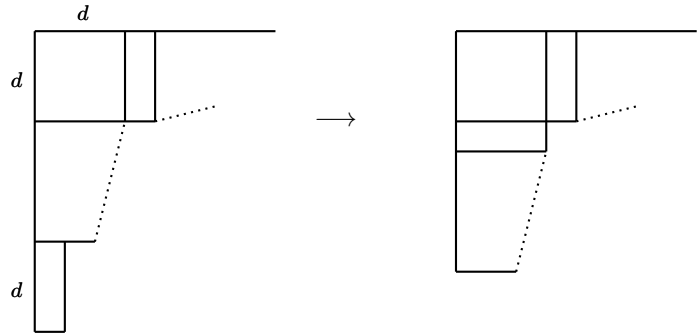

Now consider partitions whose Ferrers diagrams are of the type shown in the right-hand side of Figure \ref{fig1} with parts $1$ allowed and let $a(n)$ be the number of such partitions of $n$.  These partitions can be divided into two groups: partitions with and without parts $1$. If such a partition of $n$ has at least one part $1$, then deleting one part $1$ gives a partition of $n-1$. Thus, the number of partitions of $n$ of the type shown in the right-hand-side with no parts $1$ equals $a(n)-a(n-1)$. 

Also, because of the length $d$ row below the Durfee square and the height $d$ column to the right of the Durfee square, the Frobenius symbol for this right-hand partition has no 0 entries.  Thus, the number of partitions of $n$ in question is the number of partitions of $n$ whose Frobenius symbol has no $0$ minus the number of partitions of $n-1$ whose Frobenius symbol has no $0$. 
\end{proof}

Examining the Ferrers diagram on the right-hand side of Figure \ref{fig1} shows that Proposition~\ref{Frob1} is equivalent to the following corollary.

\begin{corollary}
The number of partitions of $n$ with crank $0$ equals the number of partitions of $n$ whose Frobenius symbol has no $0$ and the first two entries of the bottom row differ by $1$.
\end{corollary}

Our last proof gives a combinatorial argument for \cite[Theorem 8]{hss}.  We return to $j$-Durfee rectangles and the foundational Theorem~\ref{thm2.1}.

\begin{theorem}[Hopkins, Sellers, Stanton]
The number of partitions $\lambda$ of $n$ with $\crank(\lambda) \ge j$ equals the number of partitions of $n-j$ whose Frobenius symbol has no $j$ in its top row.
\end{theorem}

\begin{proof}
By Theorem~\ref{thm2.1}, we know that partitions $\lambda$ of $n$ with $\crank(\lambda) \ge j$ are in bijection with partitions of $n$ with at least $d+j$ parts $1$ where $d$ is the parameter of its $j$-Durfee rectangle. 
Modify $\lambda$ by deleting $d+j$ parts $1$ and increasing each of the $d$ largest parts of $\lambda$ by $1$. The resulting $\lambda'$ is therefore a partition of $n-j$ which satisfies
$$
\lambda'_d\ge d+j+1, \quad \lambda'_{d+1}\le d+j.
$$
Thus the top entries in columns $d$ and $d+1$ of the Frobenius symbol of $\lambda'$ are at least $j+1$ and at most $j-1$, respectively. This shows that there are no occurences of $j$ in the top row of the Frobenius symbol.  See the Figure \ref{fig2}.
\end{proof}

\begin{figure}[ht]
\begin{tikzpicture}[scale=0.8]
\draw[thick, black] (0, 1) -- (4, 1);
\draw[thick, black] (0, 0) -- (2.5, 0);
\draw[thick, black] (0, 1) -- (0, -1.5);
\draw[black] (2, 1) -- (2, 0);
\draw[black] (2.5, 1) -- (2.5, 0);
\draw[black] (2, 0) -- (2, -0.5);

\draw[black] (0,-1.5) -- (0.5, -1.5);
\draw[black] (0.5,-1.5) -- (0.5, -1.0);
\draw[black] (0.5,-1.0) -- (1.5, -1.0);
\draw[black] (1.5,-1.0) -- (1.5, -0.5);
\draw[black] (1.5,-0.5) -- (2, -0.5);

\draw (-0.3, 0.5) node{\footnotesize{$d$}};
\draw (0.5, -0.3) node{\footnotesize{$d$}};
\draw (1.1, 1.3) node{\footnotesize{$d+j$}};

\draw[dotted, black] (0,1) -- (2,-1);
\end{tikzpicture}
\caption{The result of deleting $d+j$ parts 1 and adding 1 to each of the $d$ largest parts.}
\label{fig2}
\end{figure}
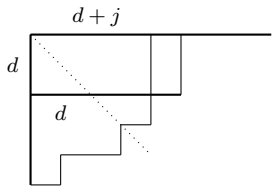

\end{document}